\documentclass[11pt,reqno]{amsart}
\usepackage{amsthm}
\usepackage{amsmath,amssymb,txfonts,tipa,amsfonts,ascmac,mathrsfs,multicol,amscd,ascmac,fancybox}
\usepackage[dvipdfmx]{graphicx}
\usepackage[pagewise]{lineno}%\linenumbers
\usepackage{amsaddr}
\usepackage[all]{xy}
\usepackage[margin=1.2in]{geometry}

\usepackage{url}
\usepackage[initials]{amsrefs}
\usepackage{amsmath,amsthm,amscd}
\usepackage{enumerate}

\usepackage{amssymb}
\usepackage{mathrsfs}
\usepackage{graphicx}

\usepackage[dvipdfmx]{color}
\usepackage{mathtools}

\usepackage[latin1]{inputenc}
\usepackage{tikz}
\usetikzlibrary{shapes,arrows}
\usetikzlibrary{positioning}
\usetikzlibrary{intersections,calc}

\makeatletter
\@namedef{subjclassname@2010}{%
  \textup{2010} Mathematics Subject Classification}
\makeatother

\numberwithin{equation}{section}

\newtheorem{theorem}{Theorem}

\newtheorem{lemma}{Lemma}[section]
\newtheorem{proposition}{Proposition}[section]

\theoremstyle{definition}
\newtheorem{definition}{Definition}[section]
\newtheorem{remark}{Remark}[section]

\newtheorem{example}{Example}[section]

\newcommand{\supp}{\operatorname*{supp}}
\newcommand{\mP}{\mathbb{P}}
\newcommand{\vep}{\varepsilon}
\newcommand{\vphi}{\varphi}

\newcommand{\LIM}{\operatorname*{LIM}}
\newcommand{\lv}{\left\lvert}
\newcommand{\lV}{\left\lVert}
\newcommand{\rv}{\right\rvert}
\newcommand{\rV}{\right\rVert}
\newcommand{\N}{\mathbb{N}}

\makeatletter
\newcommand{\opnorm}{\@ifstar\@opnorms\@opnorm}
\newcommand{\@opnorms}[1]{%
  \left|\mkern-1.5mu\left|\mkern-1.5mu\left|
   #1
  \right|\mkern-1.5mu\right|\mkern-1.5mu\right|
}
\newcommand{\@opnorm}[2][]{%
  \mathopen{#1|\mkern-1.5mu#1|\mkern-1.5mu#1|}
  #2
  \mathclose{#1|\mkern-1.5mu#1|\mkern-1.5mu#1|}
}
\makeatother
\begin{document}

\baselineskip=17pt

\title[Invariant measure for Markov operator cocycles and random mean ergodic theorem]{Invariant measure for Markov operator cocycles and random mean ergodic theorem}

\author[F. Nakamura]{Fumihiko NAKAMURA}
\address[F. Nakamura]{Faculty of Engineering, Kitami Institute of Technology, Hokkaido, 090-8507, JAPAN}
\email[F. Nakamura]{nfumihiko@mail.kitami-it.ac.jp}

\author[H. Toyokawa]{Hisayoshi TOYOKAWA}
\address[H. Toyokawa]{Department of Mathematics, Hokkaido University, Hokkaido, 060-0810, JAPAN}
\email[H. Toyokawa]{toyokawa@math.sci.hokudai.ac.jp}
%\urladdr{}

\begin{abstract}
In the present paper, we consider random invariant densities and the mean ergodic theorem for Markov operator cocycles which are applicable to quenched type random dynamical systems.
We give necessary and sufficient conditions for the existence of random invariant densities for Markov operator cocycles and establish the mean ergodic theorem for generalized linear operator cocycles over a weakly sequentially complete Banach space.
The advantage of the result is that we show the implication of weak precompactness for almost every environment to strong convergence in the global sense.
\end{abstract}


  \maketitle

\section{Introduction}
This paper concerns a cocycle generated by Markov operators, called a {\it Markov operator cocycle}. 
Let $(X,\mathcal{A},m)$ be a probability space and  $L^1 (X, m)$ the space of all $m$-integrable functions on $X$ endowed with the $L^1$-norm  $\Vert \cdot \Vert _{L^1(X)}$.
An operator $P: L^1(X, m)\to L^1 (X, m)$ is called a \emph{Markov operator} if $P$ is linear, positive (i.e., $Pf\geq 0$ $m$-almost everywhere if $f\geq 0$ $m$-almost everywhere) and $P$ preserves the mean in the sense that
\begin{align*}
\int _X Pf  dm  = \int _Xf dm \quad
\text{ for all $f \in L^1(X, m)$} .
\end{align*}
Markov operators naturally appear in the study of dynamical systems as Perron--Frobenius operators, Markov processes as integral operators with stochastic kernels of the processes, and annealed type random dynamical systems as integrations of Perron--Frobenius operators over environmental parameters (see \cite{F,LM} for details). 
  
A Markov operator cocycle is given by compositions of potentially different  Markov operators which are provided with  the environment $\{ \sigma ^n\omega\} _{n\geq 0}$ driven by a measure-preserving  transformation $\sigma$ on a probability space $(\Omega , \mathcal F, \mathbb P)$, 
\begin{align*}
\mathbb N \times \Omega \times L^1(X,m) \to  L^1(X,m):  (n, \omega , f) \mapsto P_{\sigma ^{n-1}\omega} \circ P_{\sigma ^{n-2}\omega} \circ \cdots \circ P_{\omega} f
\end{align*}
(see Definition \ref{dfn:1} for more precise description). 
Thus, it essentially possesses two kinds of randomness:
\begin{enumerate}
\item[(i)] The evolution of densities at each time are governed by Markov operators $P_\omega$;
\item[(ii)] The selection of  each Markov operators  is according to the base dynamics $\sigma$.
\end{enumerate}
Thus, by considering Markov operator cocycles, we expect to understand more complicated phenomena in multi-stochastic systems. 
The study of Markov operator cocycles follows measurable random dynamical systems in the sense of \cite{A}. We also refer to \cite{NNT}.

In this paper, we first give necessary and sufficient conditions for the existence of a random invariant measure (density) for a given Markov operator cocycle (Theorem \ref{thmA}).
As a density function of an absolutely continuous invariant measure, called an invariant density, for a deterministic dynamical system plays an important role in ergodic theory, the existence of a random invariant density for a Markov operator cocycle is also an important topic.
See \cite{BBR,Bu,Dr} for the analysis of statistical properties, including mixing rate and several limit theorems, for random dynamical systems via random invariant densities.
One of checkable criteria for the existence of an invariant density for a single Markov operator $P$ is weak precompactness of the trajectory of an initial distribution $f$: $\{P^nf\}_n$ (see \cite{T}).
For random dynamical systems, a few sufficient conditions for the existence of random invariant measures are known.
For instance, Crauel \cite{C1,C2} showed the existence by conditional expectations on the ``future'' sub-$\sigma$-algebra based on the martingale convergence theorem for random dynamical systems. Moreover, Dragi\v{c}evi\'{c} et.~al.~\cite{Dr} showed the existence of invariant measures for random transfer operators under the assumption called ``admissible cocycle.''
Our approach in the present paper succeeds to give a new criterion for the existence of random invariant densities for Markov operator cocycles, and leads to the density for random non-Markov fibered systems (see Section \ref{eg_nMFS}).

Furthermore, we prove the mean ergodic theorem for a linear operator cocycle on a general Banach space (Theorem \ref{MET}). The conventional mean ergodic theorem provides that the average of the sequence $\{P^nf\}_n$ converges in strong sense, and the limit point becomes an invariant density.
The classical mean ergodic theorem for a single linear operator by von Neumann deals only with a reflexive Banach space, and after that, Yosida and Kakutani \cite{YK} generalized the theorem to the case of a general Banach space under the assumption of weak precompactness of Ces\`{a}ro average of time evolutions.
As known in \cite{Beck}, the theorem for a linear operator cocycle is fulfilled if the Banach space is reflexive.
Then, giving an appropriate definition of weak precompactness for the cocycle, we succeeded to obtain a general result for the mean ergodic theorem of linear operator cocycles, that guarantees the existence of invariant measures for linear operator cocycles.

The organization of this paper is as follows.
In the rest of Section 1, we recall our setting and state our main theorems.
In Section 2, we introduce random non-Markov fibered systems with certain properties and apply Theorem \ref{thmA} to the systems.
In Section 3, we prepare some lemmas and relate the lift operator to skew product transformations.
Section 3, 4 and 5 are devoted to the proofs of Theorem \ref{thmA}, Proposition \ref{prop} and Theorem \ref{MET} respectively.

\subsection{Definitions: Markov operator cocycles and random invariant density}

In this subsection, we recall the definition of linear/Markov operator cocycles and their random invariant densities.

Let $D(X, m)$ be a set of all density functions, i.e., a subset of $L^1 (X, m)$ defined by
\begin{align*}
D(X, m) =\left\{f\in L^1(X, m) : f\geq 0 \ m \text{-almost everywhere and} \lV f\rV _{L^1(X)}=1 \right\}.
\end{align*}
Note that $P:L^1(X,m)\to L^1(X,m)$ is a Markov operator  if and only if  $P(D(X,m)) \subset D(X,m)$.
One of the most important examples of Markov operators is the \emph{Perron--Frobenius operator} induced by a measurable and non-singular transformation $T:X\to X$
(that is, the probability measure $m\circ T^{-1}$ is absolutely continuous with respect to $m$).
The Perron--Frobenius operator $\mathcal L_T:L^1(X,m)\to L^1(X,m)$ of $T$ is defined by 
\begin{equation}\label{eq:0913}
\int_X \mathcal L_T f\cdot g dm = \int_X f\cdot g\circ T dm \quad\quad \text{for }f\in L^1(X,m) \text{ and }g\in L^{\infty}(X,m).
\end{equation}

Let $(\Omega, \mathcal F , \mathbb{P})$ be a probability space and $\sigma:\Omega\to\Omega$ be a $\mathbb P$-preserving transformation.
Throughout the paper, we further assume $\sigma$ is invertible and ergodic (i.e., $\sigma^{-1}E=E\pmod{\mP}$ implies $E=\emptyset$ or $\Omega\pmod{\mP}$).
For a measurable space $\Sigma$,  we say that a measurable map $\Phi: \mathbb N_0 \times \Omega \times \Sigma \to \Sigma$ is a \emph{random dynamical system} on $\Sigma$ over the driving system $\sigma$ if
\begin{align*}
\varphi ^{(0)} _\omega = \mathrm{id} _{\Sigma}  \quad\text{and}\quad \varphi ^{(n+m)} _ \omega  = \varphi ^{(n)}_{ \sigma ^m\omega }\circ  \varphi ^{(m)}_\omega
\end{align*}
for each $n, m \in \mathbb N_0$ and $\omega \in \Omega$, with the notation $\varphi ^{(n)}_\omega =\Phi (n,\omega ,\cdot )$ and $\sigma \omega =\sigma (\omega )$, where $\mathbb N_0 =\mathbb N \cup \{0\}$. 
A standard reference for random dynamical systems is the monograph by Arnold \cite{A}. 
It is easy to check that 
\begin{equation}\label{eq:0220b2}
\varphi ^{(n)}_\omega = \varphi _{\sigma ^{n-1}\omega }\circ \varphi _{\sigma ^{n-2}\omega } \circ \cdots \circ \varphi _\omega 
\end{equation}
with the notation $\varphi _\omega = \Phi (1, \omega , \cdot )$.
Conversely, for each measurable map $\varphi : \Omega \times \Sigma\to \Sigma: (\omega , x) \mapsto \varphi _\omega (x)$, the measurable map $(n,\omega , x) \mapsto \varphi _\omega ^{(n)}(x)$ given by \eqref{eq:0220b2} is a random dynamical system. 
We call it a random dynamical system induced by $\varphi$ over $\sigma$, and simply denote it by $(\varphi , \sigma )$.
When $\Sigma$ is a Banach space (with its Borel measurable sets from its strong norm) and $\varphi _\omega :\Sigma \to \Sigma$ is linear $\mathbb P$-almost everywhere, $(\varphi , \sigma )$ is called a \emph{linear operator cocycle}. 
We give a formulation of Markov operators in random environments in terms of linear operator cocycles.

\begin{definition}\label{dfn:1}
A linear operator cocycle $(P, \sigma )$ induced by a measurable map $P: \Omega \times L^1(X,m) \to L^1(X,m)$ over $\sigma$ is called a \emph{Markov operator cocycle} (or a \emph{Markov operator in random environments}) if $P_\omega =P(\omega ,\cdot ): L^1 (X,m) \to L^1 (X,m)$ is a Markov operator for $\mathbb P$-almost every $\omega \in \Omega$.
\end{definition}

Suppose that $(n, \omega , f) \mapsto P^{(n)}_\omega f$ is a Markov operator cocycle induced by
$P: \Omega \times L^1(X,m) \to L^1(X,m)$ such that $P_\omega =P(\omega ,\cdot )$  is the Perron--Frobenius operator $\mathcal L_{T_\omega}$ associated with a non-singular map $T_\omega :X\to X$ for  $\mathbb P$-almost every $\omega$.
Then it follows from \eqref{eq:0913} that for $\mathbb P$-almost every $\omega\in\Omega$,
\begin{align}\label{eq:0913b}
\int_X \mathcal L_{T_\omega}^{(n)}f\cdot g dm = \int_X f\cdot g\circ T_\omega^{(n)} dm, \quad \text{for }f\in L^1(X,m) \text{ and } g\in L^{\infty}(X,m),
\end{align}
where $T_\omega^{(n)} = T_{\sigma ^{n-1}\omega }\circ T_{\sigma ^{n-2}\omega }\circ \cdots \circ T_\omega$.

Now we recall the definition of random invariant densities for Markov operator cocycles, which is the main concept of the paper.

\begin{definition}
A measurable map $h:\Omega\to L^1(X,m)$ with $h(\omega)=h_{\omega}$ is called a \emph{random invariant density}
if $h_\omega\in D(X,m)$ and $P_{\omega}h_{\omega}=h_{\sigma\omega}$ hold for $\mP$-almost every $\omega\in\Omega$.
\end{definition}

Note that the condition of invariance is equivalent to $h_{\omega}=P_{\sigma^{-1}\omega}h_{\sigma^{-1}\omega}$ for $\mP$-almost every $\omega\in\Omega$.
We give necessary and sufficient conditions for the existence of a random invariant density for a given Markov operator cocycle in Theorem \ref{thmA}.
Furthermore, in Theorem \ref{MET}, we will show the mean ergodic theorem for general linear (not necessarily Markov) operator cocycle, which guarantees the existence of a version of random invariant density.

\subsection{Definitions: The lift operator and weak precompactness}
 
We introduce our main tools in this paper: the lift operator $\mathscr{P}$ of a linear operator cocycle $(P,\sigma)$ and weak precompactness of functions in fiberwise and global sense
in order to construct a random invariant density for the linear operator cocycle.
We first prepare the Banach space of Bochner integrable functions over a Banach space $\mathfrak{X}$ (with norm $\lV\cdot\rV_{\mathfrak{X}}$) denoted by $L^1\left(\Omega,\mathfrak{X}\right)$, based on \cite{DU,IT}.
Then, we define the lift operator $\mathscr{P}$ over $L^1\left(\Omega,\mathfrak{X}\right)$ associated with the linear operator cocycle and relate it with a random invariant density.

Let us define
\begin{align*}
&\mathscr{L}^1\left(\Omega,\mathfrak{X}\right)=\left\{f:\Omega\to \mathfrak{X}, \text{strongly measurable and integrable}\right\},\\
&\mathscr{N}=\bigg\{f:\Omega\to \mathfrak{X}, \text{ strongly measurable and }\lV\vphi(\omega)\rV_{\mathfrak{X}}=0,\ \mP\text{-almost every }\omega\in\Omega \bigg\},
\end{align*}
where $f:\Omega\to \mathfrak{X}$ is called \emph{strongly measurable} provided that there exists a sequence of simple functions $f_n=\sum_{i=1}^{N}1_{F_i}v_i$ for some $N=N(n)\in\N$, $\{F_i=F_i(n):i=1,\dots,N\}\in\mathscr{F}$ and $\{v_i=v_i(n):i=1,\dots,N\}\subset \mathfrak{X}$ such that $\lim_{n\to\infty}\lV f(\omega)-f_n(\omega)\rV_{\mathfrak{X}}=0$ for $\mP$-almost every $\omega\in\Omega$.
Then we define
\begin{align*}
L^1\left(\Omega,\mathfrak{X}\right)\coloneqq\mathscr{L}^1\left(\Omega,\mathfrak{X}\right)/\mathscr{N}.
\end{align*}
It is known that if $\mathfrak{X}$ is weakly sequentially complete then so is $L^1\left(\Omega,\mathfrak{X}\right)$ and hence any weakly Cauchy sequence in $L^1(\Omega,\mathfrak{X})$ converges in $L^1\left(\Omega,\mathfrak{X}\right)$ (see \cite{Ta}).
Note that if $\mathfrak{X}=L^1(X,m)$ then $L^1\left(\Omega,L^1(X,m)\right)$ is isometric to $L^1(\Omega\times X,\mP\times m)$ (see Lemma \ref{cong}).
The space $L^1(\Omega,\mathfrak{X})$ is equipped with the usual norm $\opnorm{\cdot}_1$
given by
\begin{align*}
\opnorm{f}_1 \coloneqq\int_\Omega \left\lVert f_\omega\right\rVert_{\mathfrak{X}} d\mathbb{P}(\omega)\quad\text{for $f\in L^1(\Omega,\mathfrak{X})$}.
\end{align*}

Throughout the paper, we suppose that $P:\Omega\times\mathfrak{X}\to\mathfrak{X}$ is strongly measurable in the sense that the map for each $f\in\mathfrak{X}$, $P_{\cdot}f:\Omega\to\mathfrak{X}$ is $\mathcal{F}$-$\mathcal{B}_{\mathfrak{X}}$ measurable.
Then the lift operator of a given linear operator cocycle is defined as follows.

\begin{definition}\label{liftop}
For a linear operator cocycle $(P,\sigma)$ over a Banach space $\mathfrak{X}$ where $P_{\omega}:\mathfrak{X}\to\mathfrak{X}$ is bounded uniformly in $\omega$, the \emph{lift operator} $\mathscr{P}:L^1(\Omega,\mathfrak{X})\to L^1(\Omega,\mathfrak{X})$ is defined by
\begin{align*}
(\mathscr{P}f)(\omega)\coloneqq P_{\sigma^{-1}\omega}f_{\sigma^{-1}\omega}
\end{align*}
for $f\in L^1\left(\Omega,\mathfrak{X}\right)$ and $\mP$-almost every $\omega\in\Omega$
so that for each $n\in\N$ we have
\begin{align*}
\left(\mathscr{P}^nf\right) (\omega) =  P_{\sigma^{-n}\omega}^{(n)}f_{\sigma^{-n}\omega}
\end{align*}
for $\mP$-almost every $\omega\in\Omega$.
\end{definition}

\begin{remark}
(I)
The above lift operator is a well-defined bounded linear operator over $L^1(\Omega,\mathfrak{X})$.
Indeed, if $f:\Omega\to\mathfrak{X}$ is strongly measurable then the strong measurability of $\mathscr{P}f$ follows from Lemma A.6 in \cite{GQ}.
Moreover if $f,\tilde{f}\in L^1\left(\Omega,\mathfrak{X}\right)$ and $f-\tilde{f}\in\mathscr{N}$, then we have
\begin{align*}
\lV\mathscr{P}\left(f-\tilde{f}\right)(\omega)\rV_{\mathfrak{X}}
&=\lV P_{\sigma^{-1}\omega}\left(f_{\sigma^{-1}\omega}-\tilde{f}_{\sigma^{-1}\omega}\right)\rV_{\mathfrak{X}}\\
&\le M\lV f_{\sigma^{-1}\omega}-\tilde{f}_{\sigma^{-1}\omega}\rV_{\mathfrak{X}}=0
\end{align*}
for $\mP$-almost every $\omega\in\Omega$ where $M$ is the supremum of the operator norm of $P_{\omega}$ and $\mathscr{P}f=\mathscr{P}\tilde{f}$ $\mP$-almost everywhere.
We also have
\begin{align*}
\opnorm{\mathscr{P}f}_1 =\int_\Omega \left\lVert P_{\sigma^{-1}\omega}f_{\sigma^{-1}\omega}\right\rVert_{\mathfrak{X}} d\mathbb{P}(\omega)
\leq \int_\Omega M\left\lVert f_{\sigma^{-1}\omega}\right\rVert_{\mathfrak{X}} d\mathbb{P}(\omega) = M\opnorm{f}_1,
\end{align*}
which implies that $\mathscr{P}$ is a bounded operator.
In particular, if $\lV P_{\omega}\rV\le 1$ for $\mP$-almost every $\omega\in\Omega$ then $\mathscr{P}$ is a contraction operator over $L^1(\Omega,\mathfrak{X})$.

(II)
We note that $h\in L^1\left(\Omega,D(X,m)\right)$ is a random invariant density if and only if $\mathscr{P}h=h$ (see Proposition \ref{propA} (2) more precisely).
\end{remark}

Recall that a subset $\mathscr{F}\subset L^1(X,m)$ is called weakly precompact if for any sequence $\{f_n\}_n\subset\mathscr{F}$ there is a further subsequence $\{f_{n_k}\}_k$ which converges weakly in $L^1(X,m)$.
Now we define weak precompactness in $L^1(\Omega,\mathfrak{X})$ in two senses. 

\begin{definition}\label{wpc}
A set $\mathscr{F}\subset L^1(\Omega,\mathfrak{X})$ is called \emph{fiberwise weakly precompact} if for every sequence $\{f_n\}_n\subset \mathscr{F}$, there exists $h\in L^1(\Omega,\mathfrak{X})$ such that for $\mathbb{P}$-almost every $\omega\in\Omega$, there exists a subsequence $\{n_k\}_k=\{n_k(\omega)\}_k\subset \mathbb{N}$ such that $\{(f_{n_k})(\omega)\}_k$ converges weakly to $h(\omega)$.

A set $\mathscr{F}\subset L^1(\Omega,\mathfrak{X})$ is called \emph{globally weakly precompact} if for every sequence $\{f_n\}_n\subset \mathscr{F}$, there is a further subsequence $\{f_{n_k}\}_k$ which converges weakly in $L^1(\Omega,\mathfrak{X})$.
\end{definition}

\begin{remark}
(I) 
In Definition \ref{wpc}, the weak topology on $L^1(\Omega,\mathfrak{X})$ is induced by an abstract dual space of $L^1(\Omega,\mathfrak{X})$.
However, It is known that $L^1(\Omega,\mathfrak{X})^*\simeq L^{\infty}(\Omega,\mathfrak{X}^*)$ if and only if the dual space $\mathfrak{X}^*$ has the Radon--Nikod\'{y}m property (see Theorem 1 \cite[page 98]{DU} more precisely).

(II)
We immediately find that if $\mathscr{F}$ is globally weakly precompact in $L^1(\Omega,\mathfrak{X})$, then there is some limit point $h\in L^1(\Omega,\mathfrak{X})$ and a subsequence $\{n_k\}_k$ (independent of $\omega$) such that $f_{n_k}(\omega)$ converges weakly to $h(\omega)$. Thus, the global weak precompactness implies the fiberwise weak precompactness.
Though the converse is non-trivial, we can show the converse implication by Theorem \ref{thmA} ((6) $\Rightarrow$ (4)) for the case when $\mathfrak{X}=L^1(X,m)$ and by the proof of Theorem \ref{MET} for the case when $\mathfrak{X}$ is a weakly sequentially complete Banach space.
We also note that if $\mathfrak{X}=L^1(X,m)$ then global weak precompactness in $L^1(\Omega,\mathfrak{X})$ is equivalent to the conventional weak precompactness in $L^1(\Omega\times X,\mP\times m)$ (see also Lemma \ref{cong}).
\end{remark}

\subsection{Main theorem \ref{thmA}: Random invariant densities for Markov operator cocycles}

Let $(\Omega,\mathcal{F},\mP)$ and $(X,\mathcal{A},m)$ be probability spaces as a parameter space and a phase space, respectively, and $\sigma$ be an invertible, ergodic and $\mathbb{P}$-preserving transformation on $\Omega$.
We assume $\Omega$ and $X$ are Polish spaces, namely complete and separable metric spaces. For Theorem \ref{thmA}, we consider the case $\mathfrak{X}=L^1(X,m)$.

The first main theorem which gives necessary and sufficient conditions for the existence of a random invariant density for a Markov operator cocycle $(P,\sigma)$ is as follows.

\begin{theorem}\label{thmA}
Let $(\Omega,\mathcal{F},\mP)$ and $(X,\mathcal{A},m)$ be probability spaces with Polish spaces $\Omega$ and $X$, $\sigma$ be an invertible $\mathbb{P}$-preserving ergodic transformation on $\Omega$, $(P,\sigma)$ be a Markov operator cocycle and $\mathscr{P}$ be its lift operator over $L^1\left(\Omega,L^1(X,m)\right)$.
Then the following are equivalent:
\begin{enumerate}
\item
There exists a finite invariant density $h\in L^1\left(\Omega,L^1(X,m)\right)$ for the Markov operator cocycle such that $\mathscr{P}^{*n}1_{\supp h}(\omega,x)$ monotonically tends to $1$ for $\mP\times m$-almost every $(\omega,x)\in\Omega\times X$,
where the adjoint operator $\mathscr{P}^*$ is considered over $L^{\infty}(\Omega\times X,\mP\times m)$;
\item
For any Banach limit $\LIM$, a set function defined by
\begin{align}\label{eqbl}
\mu_{\omega} (E) \coloneqq \LIM \left(\left\{ \int_E  P_{\sigma^{-n}\omega}^{(n)}1_X dm \right\}_n\right)
\end{align}
is an absolutely continuous probability measure on $(X,\mathcal{A})$ with $P_{\omega}\frac{d\mu_{\omega}}{dm}$ $=\frac{d\mu_{\sigma\omega}}{dm}$ for $\mP$-almost every $\omega\in\Omega$;
\item
$\left\{\mathscr{P}^n1_{\Omega\times X}\right\}_n\subset L^1(\Omega\times X,\mP\times m)$ is weakly precompact;
\item 
$\mathscr{P}$ is weakly almost periodic (i.e. $\left\{\mathscr{P}^n f\right\}_n\subset L^1(\Omega\times X,\mP\times m)$ is weakly precompact for any $f\in L^1(\Omega\times X,\mP\times m)$);
\item
For $\mP$-almost every $\omega\in\Omega$, it holds that for any $f\in L^1(X,m)$
\begin{align*}
\lim_{n\to\infty}\frac{1}{n}\sum_{i=0}^{n-1} P_{\sigma^{-n}\omega}^{(n)}f
\end{align*}
exists in strong sense;
\item
For $\mP$-almost every $\omega\in\Omega$, $\{ P_{\sigma^{-n}\omega}^{(n)}1_X\}_n\subset L^1(X,m)$ is weakly precompact;
\item
For $\mP$-almost every $\omega\in\Omega$, $\{ P_{\sigma^{-n}\omega}^{(n)}f\}_n\subset L^1(X,m)$ is weakly precompact for any $f\in L^1(X,m)$.
\end{enumerate}
\end{theorem}

\begin{remark}\label{remthmA}
(I)
The important point of Theorem \ref{thmA} is to show that ``fiberwise'' weak precompactness of $\{ P_{\sigma^{-n}\omega}^{(n)}1_X\}_n$ (the condition (6)) implies the existence of a random invariant density for $\mathscr{P}$ (the condition (1)).

(II)
Several sufficient conditions for weak precompactness are known as follows (IV.8, \cite{DuSch}).
It reads that $\{ P_{\sigma^{-n}\omega}^{(n)}1_X\}_n$ is weakly precompact if one of the following three conditions holds:
\begin{enumerate}
\item[(i)]
There exists $g=g_{\omega}\in L^1_+(X,m)\coloneqq \left\{f\in L^1(X,m):f\ge0\right\}$ such that for any $n\ge1$
\begin{align*}
\lv  P_{\sigma^{-n}\omega}^{(n)}1_X(x)\rv\le g(x)\quad m\text{-almost every }x\in X;
\end{align*}
\item[(ii)]
There exists $M=M_{\omega}>0$ and $p=p_{\omega}>1$ such that 
\begin{align*}
\lV  P_{\sigma^{-n}\omega}^{(n)}1_X\rV_{L^p(m)}\le M;
\end{align*}
\item[(iii)]
$\{ P_{\sigma^{-n}\omega}^{(n)}1_X\}_n$ is uniformly integrable, namely, for any $\vep>0$ there exists $\delta>0$ such that
\begin{align*}
m(A)<\delta\quad\text{implies}\quad\int_A P_{\sigma^{-n}\omega}^{(n)}1_Xdm<\vep\quad\text{for all }n\ge1.
\end{align*}
\end{enumerate}

(III)
For the use of Banach limits to construct invariant measures for single Markov operators, we refer to \cite{DS,S,T}.
Our result is the generalization of them to Markov operator cocycles.
\end{remark}

We also give a sufficient condition for weak precompactness of $\{P_{\sigma^{-n}\omega}^{(n)}g\}_n$ for $g\in L^1_+(X,m)$,
which is simple generalization of Proposition 3 in \cite{DS} to the Markov operator cocycles case.

\begin{proposition}\label{prop}
Suppose that for $\mP$-almost every $\omega\in\Omega$ there is a random invariant density $h_{\omega}\in D(X,m)$ such that $F_{\alpha,n}\coloneqq\{h_{\sigma^{-n}\omega}>\alpha\}$ satisfies $m(F_{\alpha,n})\nearrow 1$ as $\alpha\searrow 0$ uniformly in $n\in\N$.
Then, for $\mP$-almost every $\omega\in\Omega$, the set
$\{P_{\sigma^{-n}\omega}^{(n)}g\}_n$ is uniformly integrable (and hence weakly precompact) for any $g\in L^1_+(X,m)$.
\end{proposition}

\subsection{Main theorem B: Mean ergodic theorem for linear operator cocycles}

We now state the mean ergodic theorem for linear operator cocycle on general Banach space. 

Let $(\Omega,\mathcal{F},\mathbb{P})$ be a probability space and $\sigma$ be an invertible, $\mathbb{P}$-preserving and ergodic transformation on $\Omega$.
Let $\mathfrak{X}$ be a weakly sequentially complete Banach space and $P: \Omega\times \mathfrak{X}\to \mathfrak{X}$ a linear operator cocycle which is contraction almost everywhere.
Again $L^1(\Omega,\mathfrak{X})$ denotes the set of all $L^1$-functions from $\Omega$ to $\mathfrak{X}$ equipped with the norm $\opnorm{\cdot}_1$
given by
\begin{align*}
\opnorm{f}_1 \coloneqq\int_\Omega \left\lVert f_\omega\right\rVert_{\mathfrak{X}} d\mathbb{P}(\omega)\quad\text{for $f\in L^1(\Omega,\mathfrak{X})$}.
\end{align*}
We define the operator $\mathscr{A}^n$ meaning the average of $\mathscr{P}^n$ by
\begin{align*}
(\mathscr{A}^nf)(\omega)\coloneqq\frac{1}{n}\sum_{k=0}^{n-1} (\mathscr{P}^k f)(\omega)=\frac{1}{n}\sum_{k=0}^{n-1} P_{\sigma^{-k}\omega}^{(k)}f_{\sigma^{-k}\omega}
\end{align*}
for $f\in L^1(\Omega,\mathfrak{X})$ and $\mP$-almost every $\omega\in\Omega$.
Recall that if a sequence $\{\mathscr{A}^nf\}_n$ is fiberwise weakly precompact for $f\in L^1(\Omega,\mathfrak{X})$, then there exists $h\in L^1(\Omega,\mathfrak{X})$ such that for $\mathbb{P}$-almost every $\omega\in\Omega$, there exists a subsequence $\{n_k\}_k\subset \mathbb{N}$, $n_k=n_k(\omega,f)$, such that $(\mathscr{A}^{n_k}f)(\omega)$ converges weakly to $h(\omega)$ for $\mP$-almost every $\omega\in\Omega$.

The second main theorem in this paper is as follows.

\begin{theorem}\label{MET}
Let $\mathfrak{X}$ be a (real or complex) weakly sequentially complete Banach space, $\sigma$ an invertible $\mathbb{P}$-preserving ergodic transformation over the probability space $(\Omega,\mathcal{F},\mathbb{P})$, and $P_\omega$ a linear operator which maps $\mathfrak{X}$ into itself. Assume that $\left\lVert P_\omega\right\rVert\leq 1$ for $\mathbb{P}$-almost every $\omega\in\Omega$ and 
$\{\mathscr{A}^nf\}_n$ is fiberwise weakly precompact for any $f\in L^1(\Omega,\mathfrak{X})$.
Then there exists $h\in L^1(\Omega,\mathfrak{X})$ such that
\begin{align*}
\lim_{n\to\infty}\left\lVert (\mathscr{A}^nf)(\omega)-h(\omega)\right\rVert_{\mathfrak{X}}=0,
\end{align*}
and $P_\omega h_\omega=h_{\sigma\omega}$ for $\mathbb{P}$-almost every $\omega\in\Omega$.
\end{theorem}

\begin{remark}\label{rem}
(I)
As a consequence of Theorem \ref{MET}, we immediately find that the fiberwise weak precompactness of $\{\mathscr{A}^nf\}_n$ implies globally weak precompactness for any $f\in L^1(\Omega,\mathfrak{X})$. Furthermore,
the converse of Theorem \ref{MET} (i.e., the strong convergence of $\mathscr{A}^nf$ implies the fiberwise weak preconpactness of $\{\mathscr{A}^nf\}_n$) is also obvious.

(II)
When $\mathfrak{X}=L^1(X,m)$, the existence of a random invariant density is derived from Theorem \ref{thmA}.
In that case, the relation $L^1(\Omega,L^1(X,m)) \cong L^1(\Omega\times X,\mP\times m)$ and the lift operator $\mathscr{P}$ work well as in Theorem \ref{thmA} and its proof.
However, for a general Banach space, we only have the implication of fiberwise weak precompactness to the existence and strong convergence of the limit as in Theorem \ref{MET}.
\end{remark}

\section{Examples}\label{eg_nMFS}

In this section, we introduce random non-Markov fibered systems with countable partitions satisfying bounded distortion property, and apply Theorem \ref{thmA} to the systems in Proposition \ref{exampleprop}.
For classical deterministic transformations with fibered structures satisfying bounded distortion, we refer to \cite{Sch}.
We emphasize that the fibered systems considered in this section are defined over countable (not necessarily finite) sets in standard probability spaces so that the systems include random expanding maps and random Lasota--Yorke maps (see also \cite{Dr,D2}).
For these random maps, Proposition \ref{exampleprop} guarantees the existence of random invariant density.
Throughout this section, we let $(X,\mathcal{A},m)$ and $(\Omega,\mathcal{F},\mP)$ be standard probability spaces as a phase space and a parameter space, respectively.

\begin{example}
Let $I$ be a non-empty countable set and $T_{\omega}:X\to X$ a non-singular transformation for each $\omega\in\Omega$.
We consider \emph{random fibered systems} with the following properties.
\begin{enumerate}
\item
There exists an $\omega$-independent partition of $X$ (up to $m$-null sets), $\{X_i\}_{i\in I}$ such that $T_{\omega}\rvert_{X_i}:X_i\to T_{\omega}X_i$ is invertible for each $i\in I$.
We write $X_{i_1i_2\dots i_n}^{\omega}=X_{i_1}\cap T_{\sigma^{-n}\omega}^{-1}X_{i_2}\cap\dots\cap T_{\sigma^{-n}\omega}^{-1}\circ\cdots\circ T_{\sigma^{-1}\omega}^{-1}X_{i_n}$ and $V_{i_1\dots i_n}^{\omega}:T_{\sigma^{-n}\omega}^{(n)}X_{i_1\dots i_n}^{\omega}\to X_{i_1\dots i_n}^{\omega}$ for the local inverses of $T_{\sigma^{-n}\omega}^{(n)}$;
\item
The partition satisfies generator condition, that is, for $\mP$-almost every $\omega\in\Omega$, the $\sigma$-algebra generated by cylinder sets $\{X_{i_1i_2\dots i_n}^{\omega}:n\ge1,\ i_1,\dots,i_n\in I\}$ is $\mathcal{B}\pmod{m}$;
\item
Transformations satisfy bounded distortion property, namely, there exists $C\ge1$ such that for $\mP$-almost every $\omega\in\Omega$ and any non-empty cylinder $X_{i_1\dots i_n}^{\omega}$, we have
\begin{align*}
\underset{x\in T_{\sigma^{-n}\omega}^{(n)}X_{i_1\dots i_n}^{\omega}}{\operatorname*{ess\ }\sup}\frac{d\left(m\circ V_{i_1\dots i_n}^{\omega}\right)}{dm}(x)
\le
C \underset{x\in T_{\sigma^{-n}\omega}^{(n)}X_{i_1\dots i_n}^{\omega}}{\operatorname*{ess\ }\inf}\frac{d\left(m\circ V_{i_1\dots i_n}^{\omega}\right)}{dm}(x)
\end{align*}
\item
There exists a constant $c>0$ such that for any $n\ge1$ and any cylinder $X_{i_1\dots i_n}^{\omega}$, we have
\begin{align*}
m\left(T_{\sigma^{-n}\omega}^{(n)}X_{i_1\dots i_n}^{\omega}\right)>c.
\end{align*}
Note that if $\# I<\infty$ and $T_{\omega}$ has Markov structure for $\mP$-almost every $\omega\in\Omega$ then this property is always fulfilled.
\end{enumerate}
Then the following estimate is easy to deduce.

\begin{lemma}\label{lem:ex1}
For $\mP$-almost every $\omega\in\Omega$ and any non-empty cylinder $X_{i_1\dots i_n}^{\omega}$, we have for $m$-almost every $x\in X$
\begin{align*}
\frac{d\left(m\circ V_{i_1\dots i_n}^{\omega}\right)}{dm}(x)
\le
\frac{Cm\left(X_{i_1\dots i_n}^{\omega}\right)}{m\left(T_{\sigma^{-n}\omega}^{(n)}X_{i_1\dots i_n}^{\omega}\right)}
\end{align*}
\end{lemma}

\begin{proof}
By bounded distortion (3), we have
\begin{align*}
\frac{d\left(m\circ V_{i_1\dots i_n}^{\omega}\right)}{dm}m\left(T_{\sigma^{-n}\omega}^{(n)}X_{i_1\dots i_n}^{\omega}\right)
&\le C\int_{T_{\sigma^{-n}\omega}^{(n)}X_{i_1\dots i_n}^{\omega}}\frac{d\left(m\circ V_{i_1\dots i_n}^{\omega}\right)}{dm}dm\\
&= Cm\left(X_{i_1\dots i_n}^{\omega}\right)
\end{align*}
and proof is completed.
\end{proof}

We can apply Theorem \ref{thmA} to (quenched type) random fibered systems satisfying (1)--(4) as follows.

\begin{proposition}\label{exampleprop}
Let $(T,\sigma)$ be a random fibered system with the above conditions (1)--(4).
Then, the transfer operator cocycle $\{P_{\sigma^{-n}\omega}^{(n)}1_X\}$ defined as in (\ref{eq:0913b}) is weakly precompact for $\mP$-almost every $\omega\in\Omega$.
That is, there exists a random invariant density in the sense of Theorem \ref{thmA}.
\end{proposition}

\begin{remark}
As in Proposition \ref{exampleprop}, for the existence of a random invariant density, we only need the hypothesis (1)--(4) that is a different approach from that of \cite{Dr} (i.e., a spectral approach); for instance, we need not assume countability of $\{T_{\omega}:\omega\in\Omega\}$ nor quasi-compactness of a transfer operator cocycle over some appropriate Banach space, while such strong assumptions are required in establishing further limit theorems as in \cite{Dr}.
It seems not direct to analyze the regularity of a random invariant density as in Theorem 28.1.5 in \cite{Sch} or further statistical properties from the assumption (1)--(4).
This would be our next step.
\end{remark}

\noindent
{\bf Proof of Proposition \ref{exampleprop}}.
By the generator condition (2), we only show uniform integrability of $P_{\sigma^{-n}\omega}^{(n)}1_X$ over cylinder sets (see (II)-(iii) in Remark \ref{remthmA}).
For $\vep>0$, set $\delta=\vep\inf_{n,\omega}m(T_{\sigma^{-n}\omega}^{(n)}X)/C$ where $C>0$ is given in the condition (3).
Fix $\omega\in\Omega$, $A=X_{k_1\dots k_q}^{\omega}$ a cylinder set of length $q$ with $m(A)<\delta$ and $p\ge1$.
Write
\begin{align*}
m\left(T_{\sigma^{-p}\omega}^{(-p)}A\right)
&=\sum_{l_1,\dots,l_p}\int_{V_{l_1\dots l_p}^{\omega}\left(X_{k_1\dots k_q}^{\omega}\cap T_{\sigma^{-p}\omega}^{(p)}X_{l_1\dots l_p}^{\omega}\right)}1_Xdm\\
&=\sum_{l_1,\dots,l_p}\int_{V_{k_1\dots k_q}^{\omega}T_{\sigma^{-q}\omega}^{(q)}\left(X_{k_1\dots k_q}^{\omega}\cap T_{\sigma^{-p}\omega}^{(p)}X_{l_1\dots l_p}^{\omega}\right)}\frac{d\left(m\circ V_{l_1\dots l_p}^{\omega}\right)}{dm}dm\\
&=\sum_{l_1,\dots,l_p}\int_{T_{\sigma^{-(p+q)}\omega}^{(p+q)}\left(X_{k_1\dots k_ql_1\dots l_p}^{\omega}\right)}\frac{d\left(m\circ V_{l_1\dots l_p}^{\omega}\right)}{dm}\circ V_{k_1\dots k_q}^{\omega}\cdot\frac{d\left(m\circ V_{k_1\dots k_q}^{\omega}\right)}{dm}dm\\
&\le \frac{Cm\left(X_{k_1\dots k_q}^{\omega}\right)}{m\left(T_{\sigma^{-q}\omega}^{(q)}X_{k_1\dots k_q}\right)}
\end{align*}
where we used Lemma \ref{lem:ex1} for the last inequality.
The choice of $\delta$ and $A$ implies
\begin{align*}
m\left(T_{\sigma^{-p}\omega}^{(-p)}A\right)<\vep
\end{align*}
and we prove the proposition.

\begin{remark}
We give a small remark on the potential of further applications.
For random Liverani--Saussol--Vaienti maps, dealt with in \cite{BB} for instance, which has a common indifferent fixed point, the condition (3) or (4) above are not satisfied.
Thus, Theorem \ref{thmA} may not be applicable to the system, depending on the parameter of tangency at the indifferent fixed point.
For the cases of failure, induced transformations or jump transformations (see \cite{T}) in quenched sense might work.
Furthermore, by constructing induced operators and jump operators for Markov operator cocycles, and applying Theorem \ref{thmA} to the induced/jump operators, we expect to deduce that more general Markov operator cocycles admit random invariant (locally integrable) densities.
We will investigate it in the other paper.
\end{remark}

\end{example}

\section{Preliminaries}

In this section, we prepare some basic backgrounds to prove our main theorems.
We first show the following isometric isomorphism between $L^1(\Omega,L^1$ $(X,m))$ and $L^1(\Omega\times X,\mP\times m)$, that identifies a random invariant density $h\in L^1(\Omega,L^1(X,m))$ as a function in $L^1(\Omega\times X,\mP\times m)$.

\begin{lemma}\label{cong}
$L^1(\Omega,L^1(X,m))\cong L^1(\Omega\times X,\mP\times m)$ holds.
\end{lemma}

\begin{proof}
By separability of $L^1(\Omega,L^1(X,m))$ and $L^1(\Omega\times X,\mP\times m)$, we only show countable bases $\{u_n\}_n\in L^1(\Omega,L^1(X,m))$ are naturally embedded densely in $L^1(\Omega\times X,\mP\times m)$ and vice versa.
As countable bases in $L^1(\Omega,L^1$ $(X,m))$, we take simple functions $u_n=\sum_{k=1}^{N(n)}\tilde{u}_k^{(n)}1_{\Omega_k^{(n)}}$ where $\{\Omega_k^{(n)}:k=1,\dots,N(n),n\in\N\}$ generates $\mathcal{F}$ and $\tilde{u}_k$ are also dense simple functions in $L^1(X,m)$.
Thus $u_n$ forms $u_n=\sum_{k,l}\alpha_l1_{\Omega_k^{(n)}}1_{X_l^{(n)}}$ for $\alpha_l\in\mathbb{Q}$.
This generates $L^1(\Omega\times X,\mP\times m)$ and the converse can be checked by the same manner.
\end{proof}

From the above Lemma \ref{cong}, we have
\begin{align*}
L^1\left(\Omega,D(X,m)\right)\subset L^1\left(\Omega,L^1(X,m)\right)\cong L^1\left(\Omega\times X,\mP\times m\right)
\end{align*}
and we frequently identify $h\in L^1\left(\Omega,D(X,m)\right)$ as a function in $L^1(\Omega\times X,\mP\times m)$.
We can characterize a random invariant density $h\in L^1(\Omega,L^1$ $(X,m))$, a fixed point of $\mathscr{P}$, as a function of $L^1(\Omega\times X,\mP\times m)$.

\begin{proposition}\label{propA}
The following statements are true 
under the assumption that $\sigma$ is invertible and ergodic:
\begin{enumerate}
\item
The lift operator $\mathscr{P}$ can be naturally identified with a Markov operator over $L^1(\Omega\times X,\mP\times m)$ (this operator is also denoted by the same symbol);
\item
$h\in L^1\left(\Omega,D(X,m)\right)$ is a random invariant density if and only if $\mathscr{P}h=h$ as a function of $D(\Omega\times X,\mP\times m)$;
\item
the following diagram commutes:
\begin{align*}
\xymatrix{
L^1(\Omega,L^1(X,m))\ar[r]^-{\mathscr{P}}\ar[d]_-{\iota}\ar@{}[rd]|{\circlearrowright}&L^1(\Omega,L^1(X,m))\ar[d]^-{\iota}\\
L^1(\Omega\times X,\mP\times m)\ar[r]_-{\mathscr{P}}&L^1(\Omega\times X,\mP\times m)
}
\end{align*}
where
$\iota$ is the isometry arises in Lemma \ref{cong}.
\end{enumerate}
\end{proposition}

\begin{proof}
(1):
Linearity and positivity of $\mathscr{P}$ are obvious and we only show Markov property.
For $\vphi\in L^1_+(\Omega\times X,\mP\times m)$, $\vphi_{\omega}\in L^1_+(X,m)$ for $\mP$-almost every $\omega\in\Omega$.
Then we have
\begin{align*}
\lV\mathscr{P}\vphi\rV_{L^1(\Omega\times X)}
&= \int_{\Omega}\int_X \lv P_{\sigma^{-1}\omega}\vphi_{\sigma^{-1}\omega}\rv dmd\mP(\omega)\\
&= \int_{\Omega}\int_X\lv\vphi_{\sigma^{-1}\omega}\rv dmd\mP(\omega)\\
&= \lV\vphi\rV_{L^1(\Omega\times X)}
\end{align*}
since $\sigma$ is $\mP$-preserving and $P_{\omega}$ is Markov for $\mP$-almost every $\omega\in\Omega$.

(2)
(Necessity):
Let $h\in D(\Omega\times X,\mP\times m)$ a fixed point of $\mathscr{P}$.
Then a measurable map $\omega\mapsto h_{\omega}\in L^1(X,m)$ satisfies $P_{\sigma^{-1}\omega}h_{\sigma^{-1}\omega}=h_{\omega}$ for $\mP$-almost every $\omega\in\Omega$.
We then check that $h_{\omega}\in D(X,m)$ for $\mathbb{P}$-almost every $\omega\in\Omega$.
If we set $\Omega_{1}\coloneqq\{\omega\in\Omega:\int_Xh_{\omega}dm<1\}$, then for $\omega\in\Omega_{1}$ we have
\begin{align*}
\int_Xh_{\sigma^{-1}\omega}dm=\int_XP_{\sigma^{-1}\omega}h_{\sigma^{-1}\omega}dm
=\int_Xh_{\omega}dm<1
\end{align*}
since $P_{\omega}$ is Markov and $P_{\sigma^{-1}\omega}h_{\sigma^{-1}\omega}=h_{\omega}$ for $\mP$-almost every $\omega\in\Omega$.
This implies that $\omega\in\Omega_{1}$ if and only if $\sigma^{-1}\omega\in\Omega_{1}$ and hence $\Omega_{1}$ is a $\sigma$-invariant set.
Then, by ergodicity of $(\sigma,\mP)$, $\Omega_{1}=\emptyset$ or $\Omega\pmod{\mP}$, but $\Omega_{1}$ cannot be the entire space $\Omega$ since $\int_{\Omega\times X}hd(\mP\times m)=1$.
Therefore, we conclude $\Omega_{1}=\emptyset\pmod{\mP}$ and we can also see that $\Omega_{2}\coloneqq\{\omega\in\Omega:\int_Xh_{\omega}dm>1\}$ is a $\mP$-null set as the same way.

(Sufficiency):
If a measurable map $\omega\mapsto h_{\omega}\in D(X,m)$ is a random invariant density, then $h\in L^1(\Omega\times X,\mP\times m)$; $h(\omega,x)=h_{\omega}(x)$ satisfies
\begin{align*}
\int_{\Omega\times X}hd(\mP\times m)=\int_{\Omega}\lV h_{\omega}\rV_{L^1(X)}d\mP=1
\end{align*}
and
\begin{align*}
(\mathscr{P}h)(\omega,x)=P_{\sigma^{-1}\omega}h_{\sigma^{-1}\omega}(x)=h_{\omega}(x)=h(\omega,x)
\end{align*}
for $\mP\times m$-almost every $(\omega,x)\in\Omega\times X$.

(3):
This follows from Lemma \ref{cong} and very definitions of $\mathscr{P}$ and $\iota$.
This completes the proof.
\end{proof}

An important example of the lift operator $\mathscr{P}$ of a Markov operator cocycle is the Perron--Frobenius operator of a skew product transformation of random transformations.

\begin{proposition}
Let $\Theta:\Omega\times X\to\Omega\times X$ be a $\mP\times m$ non-singular skew product transformation defined by
\begin{align*}
\Theta(\omega,x)=\left(\sigma\omega,T_{\omega}x\right)
\end{align*}
where $T_{\omega}:X\to X$ is a non-singular transformation for $\omega\in\Omega$.
Then the lift operator associated with the cocyle of $\mathcal{L}_{\omega}$ the Perron--Frobenius operator of $T_{\omega}$ is the Perron--Frobenius operator of $\Theta$.
\end{proposition}

\begin{proof}
Let $F\in\mathcal{F}$ and $A\in\mathcal{A}$.
Then $\Theta^{-1}(F\times A)=\bigcup_{\omega\in\sigma^{-1} F}\{\omega\}\times T_{\omega}^{-1}A$.
For the Perron--Frobenius operator of $\Theta$, say $\mathcal{L}_{\Theta}$, write for any $\vphi\in L^1(\Omega\times X,\mP\times m)$
\begin{align*}
\int_{F\times A}\mathcal{L}_{\Theta}\vphi d(\mP\times m)
&= \int_{\Theta^{-1}\left(F\times A\right)}\vphi d(\mP\times m)\\
&= \int_{\sigma^{-1} F}\int_{T_{\omega}^{-1}A}\vphi_{\omega}dmd\mP(\omega)\\
&= \int_F\int_A\mathcal{L}_{\sigma^{-1}\omega}\vphi_{\sigma^{-1}\omega}dmd\mP(\omega)\\
&= \int_{F\times A}\mathscr{P}\vphi d(\mP\times m).
\end{align*}
Any measurable set in $\Omega\times X$ is approximated by a rectangle set and proof is completed.
\end{proof}

Equivalent conditions for the existence of a finite invariant density for a single Markov operator $P$ over $L^1(X,m)$ was given in \cite{T}.
If we replace $P$ over $L^1(X,m)$ by $\mathscr{P}$ over $L^1(\Omega\times X,\mP\times m)$ in \cite{T}, the following lemma is valid.

\begin{lemma}[Theorem 3.1 in \cite{T}]\label{X}
Let $\mathscr{P}$ be the lift operator for a given Markov operator cocycle $(P,\sigma)$.
Then the following are equivalent:
\begin{enumerate}
\item
There exists a finite invariant density $h\in L^1\left(\Omega\times X,\mP\times m)\right)$ for $\mathscr{P}$ such that $\mathscr{P}^{*n}1_{\supp h}(\omega,x)$ monotonically tends to $1$ for $\mP\times m$-almost every $(\omega,x)\in\Omega\times X$;
\item
$\{\mathscr{P}^n1_{\Omega\times X}\}_n$ is weakly precompact in $L^1(\Omega\times X,\mP\times m)$;
\item
$\mathscr{P}$ is weakly almost periodic i.e.,
for each $f\in L^1(\Omega\times X,\mP\times m)$, $\{\mathscr{P}^nf\}_n$ is weakly precompact;
\item
$\mathscr{P}$ is mean ergodic i.e.,
for each $f\in L^1(\Omega\times X,\mP\times m)$,
\begin{align*}
\lim_{n\to\infty}\frac{1}{n}\sum_{i=0}^{n-1}\mathscr{P}^if
\end{align*}
exists in strong sense.
\end{enumerate}
\end{lemma}

\section{Proof of Theorem \ref{thmA}}

The proof of our theorem is to relate ``fiberwise'' weak precompactness of $\{ P_{\sigma^{-n}\omega}^{(n)}1_X\}_n$ and ``global'' weak precompactness of $\left\{\mathscr{P}^n1_{\Omega\times X}\right\}_n$ and impose some fiberwise properties of $P_{\omega}$ into the properties of $\mathscr{P}$ based on \cite{G}.
We prepare a sequence of lemmas in order to prove Theorem \ref{thmA}.
Note that by Lemma \ref{X} the implications (1) $\Leftrightarrow$ (3) $\Leftrightarrow$ (4) in Theorem \ref{thmA} are obvious.

\begin{lemma}\label{XX}
The condition that $\{ P_{\sigma^{-n}\omega}^{(n)}1_X\}_n$ is weakly precompact for $\mP$-almost every $\omega\in\Omega$ implies
the condition that $\{\omega\mapsto  P_{\sigma^{-n}\omega}^{(n)}1_X\}_n$ is weakly precompact in $L^1(\Omega,L^1(X,m))$.
\end{lemma}

\begin{proof}
Suppose $\{ P_{\sigma^{-n}\omega}^{(n)}1_X\}_n$ is weakly precompact for $\mathbb{P}$-almost every $\omega\in\Omega$.
Since $P_{\omega}$ is Markov for $\mathbb{P}$-almost every $\omega\in\Omega$, it holds for any $n\ge1$ and $E\in\mathcal{F}$
\begin{align*}
\int_E\left\lVert P_{\sigma^{-n}\omega}^{(n)}1_X\right\rVert_{L^1(X)}d\mathbb{P}(\omega)=\int_E1d\mathbb{P}=\mathbb{P}(E)
\end{align*}
and hence $\{\omega\mapsto  P_{\sigma^{-n}\omega}^{(n)}1_X\}_n$ 
is bounded and uniform integrable in $L^1(\Omega,L^1(X,m))$.
Thus $\{\omega\mapsto  P_{\sigma^{-n}\omega}^{(n)}1_X\}_n$ is weakly Cauchy in $L^1(\Omega,L^1(X,m))$ from Corollary 2.4 (i) in \cite{G}.
Note that $L^1(X,m)$ is weakly sequentially complete  so that $L^1(\Omega,L^1(X,m))$ is also weakly sequentially complete (see \cite{Ta}), and hence we conclude that $\{\omega\mapsto  P_{\sigma^{-n}\omega}^{(n)}1_X\}_n$ is weakly precompact in $L^1(\Omega,L^1(X,m))$.
\end{proof}

\begin{lemma}\label{wp}
The implication (6) $\Rightarrow$ (3) is true.
\end{lemma}

\begin{proof}
By Lemma \ref{XX}, $\{\omega\mapsto P_{\sigma^{-n}\omega}^{(n)}1_X\}_n$ is weakly precompact.
Moreover, by Lemma \ref{cong}, we have
\begin{align*}
L^{\infty}(\Omega\times X,\mP\times m)\cong\left(L^1\left(\Omega,L^1(X,m)\right)\right)^*
\end{align*}
and so $\{\mathscr{P}^n1_{\Omega\times X}\}_n$ is weakly precompact in $L^1(\Omega\times X,\mP\times m)$.
\end{proof}

\begin{lemma}\label{bl}
The implication (3) $\Rightarrow$ (2) is true.
\end{lemma}

\begin{proof}
Define $\omega\mapsto\mu_{\omega}$ a map from $\Omega$ to the space of linear functional on $L^{\infty}(X,m)$ by, for each $\omega\in\Omega$ and $f\in L^{\infty}(X,m)$,
\begin{align*}
\mu_{\omega} (f) \coloneqq \LIM \left(\left\{ \int_X \left( P_{\sigma^{-n}\omega}^{(n)}\right)^* f dm \right\}_n\right)
\end{align*}
for a fixed Banach limit $\LIM$.
We will show that $\mu_{\omega}(A)\coloneqq\mu_{\omega}(1_A)$ is an $m$-absolutely continuous measure such that
\begin{align*}
P_{\omega}\frac{d\mu_{\omega}}{dm}=\frac{d\mu_{\sigma\omega}}{dm}.
\end{align*}
A set function $\mu_{\omega}:\mathcal{A}\to[0,1]$ is obviously absolutely continuous with respect to $m$ for each $\omega$.
Hence we show $\sigma$-additivity and $P_{\omega}$-invariance of $\mu_{\omega}$.

To see countable additivity of $\mu_{\omega}$, let $A=\bigcup_{i\ge1}A_i$ a disjoint union of measurable sets.
Then we have
\begin{align*}
\mu_{\omega} (A) &= \LIM \left(\left\{ \int_{\bigcup_{i\ge1}A_i}  P_{\sigma^{-n}\omega}^{(n)} 1_X dm \right\}_n\right)\\
&= \LIM \left(\left\{ \sum_{i=1}^{\infty}\int_{A_i}  P_{\sigma^{-n}\omega}^{(n)} 1_X dm \right\}_n\right)
\end{align*}
and weak precompactness of $\{  P_{\sigma^{-n}\omega}^{(n)}1_X \}_n$ implies
\begin{align*}
\mu_{\omega}(A) = \LIM\left(\sum_{i=1}^{\infty}\left\{ \int_{A_i}  P_{\sigma^{-n}\omega}^{(n)} 1_X dm \right\}_n\right)
=\sum_{i=1}^{\infty} \mu_{\omega}(A_i).
\end{align*}

Next we show for any $A\in\mathcal{A}$,
\begin{align*}
\int_A P_{\omega} \frac{d\mu_{\omega}}{dm} dm = \mu_{\sigma\omega} (A).
\end{align*}
This equality follows from the calculation
\begin{align*}
\mu_{\omega}\left( P_{\omega}^*1_A \right)
&= \LIM\left(\left\{ \int_X \left( P_{\sigma^{-n}\omega}^{(n)}\right)^* \circ P_{\omega}^* 1_A \right\}_n\right)\\
&= \LIM\left(\left\{ \int_X \left(P_{\sigma^{-n}\omega}^{(n+1)}\right)^* 1_A \right\}_n\right)\\
& = \LIM\left(\left\{ \int_X \left(P_{\sigma^{-(n-1)}\omega}^{(n)}\right)^* 1_A \right\}_n\right)
=\mu_{\sigma\omega}(A).
\end{align*}
Thus 
\begin{align*}
P_{\omega}\frac{d\mu_{\omega}}{dm}=\frac{d\mu_{\sigma\omega}}{dm}
\end{align*}
holds for $m$-almost every $x\in X$.
\end{proof}

\begin{lemma}\label{wap}
The implication (6) $\Rightarrow$ (7) is true, i.e., (6) $\Leftrightarrow$ (7).
\end{lemma}

\begin{proof}
From the condition (6), there is a set $\Omega_0\in\mathcal{F}$ with $\mP(\Omega\setminus\Omega_0)=0$ such that for any $\omega\in\Omega_0$, $\{P^{(n)}_{\sigma^{-n}\omega}f\}_n$ is weakly precompact.
Now we fix $\omega\in\Omega_0$.
Note that weak precompactness of $\{P_{\sigma^{-n}\omega}^{(n)}f\}_n$ is equivalent to its uniform integrability by Remark \ref{remthmA} (III)-(iii).
We only have to show for any $\vep>0$ and $f\in L^1_+(X,m)$, there exists $\delta=\delta(\vep,f)>0$ such that for each $n\in\N$, $m(A)<\delta$ implies $\int_A P_{\sigma^{-n}\omega}^{(n)}fdm<\vep$
since $\lvert P_{\sigma^{-n}\omega}^{(n)}f\rvert\le P_{\sigma^{-n}\omega}^{(n)}\lv f\rv$ for any $f\in L^1(X,m)$.
For fixed $\vep>0$ and $f\in L^1_+(X,m)$, we take a simple function $\sum_{i=1}^K\alpha_i1_{A_i}$ such that $\lVert f-\sum_{i=1}^K\alpha_i1_{A_i}\rVert_{L^1(X)}<\vep/2$.
For this $K=K(\vep,f)$ and $\alpha\coloneqq\max_{1\le i\le K}\alpha_i$, it follows from our assumption that we can choose $\delta>0$ such that
\begin{align*}
m(A)<\delta \text{ implies } \int_A P_{\sigma^{-n}\omega}^{(n)}1_Xdm<\frac{\vep}{2K\alpha}.
\end{align*}
Therefore we have for any $A\in\mathcal{A}$ with $m(A)<\delta$
\begin{align*}
\int_A P_{\sigma^{-n}\omega}^{(n)}fdm
&\le \lV  P_{\sigma^{-n}\omega}^{(n)}\left(f-\sum_{i=1}^K\alpha_i1_{A_i}\right)\rV_{L^1(X)} + \int_A P_{\sigma^{-n}\omega}^{(n)}\sum_{i=1}^K\alpha_i1_{A_i}dm<\vep
\end{align*}
and our claim is proven.
\end{proof}

\noindent
{\bf Proof of Theorem \ref{thmA}}.
(1) $\Leftrightarrow$ (3) $\Leftrightarrow$ (4) $\Leftrightarrow$ (6) $\Leftrightarrow$ (7): See Lemma \ref{X}, \ref{wp} and \ref{wap}.

(3) $\Rightarrow$ (2): This follows from Lemma \ref{bl}.

(2) $\Rightarrow$ (6): Suppose the set function defined by (\ref{eqbl}) satisfies countable additivity for $\mP$-almost every $\omega\in\Omega$.
Then we have for $\mP$-almost every $\omega\in\Omega$ 
\begin{align*}
\sup_n \lv\int_E P_{\sigma^{-n}\omega}^{(n)}1_Xdm\rv\to 0
\quad\text{as}\quad
m(E)\to0
\end{align*}
and hence $\{P_{\sigma^{-n}\omega}^{(n)}1_X\}_n$ is weakly precompact.

(1) $\Rightarrow$ (5): Mean ergodicity of $\mathscr{P}$ follows from Lemma \ref{X}.
Thus, for any $\vphi\in L^1(\Omega\times X,\mP\times m)$, the following limit exists in strong sense:
\begin{align*}
\lim_{n\to\infty}\frac{1}{n}\sum_{i=0}^{n-1}\mathscr{P}^n\vphi.
\end{align*}
Looking at $\omega$-section of this limit, for $\mP$-almost every $\omega\in\Omega$, we have
\begin{align*}
\lim_{n\to\infty}\frac{1}{n}\sum_{i=0}^{n-1}P_{\sigma^{-n}\omega}^{(n)}f
\end{align*}
for any $f\in L^1(X,m)$ since the operator of taking $\omega$-section commutes with limit.

(5) $\Rightarrow$ (2): If the condition (5) is satisfied, then any Banach limit coincides with the Ces\`{a}ro average and Vitali--Hahn--Saks's theorem guarantees that the set function defined by (\ref{eqbl}) becomes a measure. This completes the proof.

\section{Proof of Proposition \ref{prop}}

\begin{proof}
Suppose there is a set $\Omega_0\in\mathcal{F}$ with $\mP(\Omega\setminus\Omega_0)=0$ such that for each $\omega\in\Omega_0$ we have the assertion in the assumption of Proposition \ref{prop} and fix an arbitrary $\omega\in\Omega_0$.
For any fixed $g\in L^1_+(X,m)$ we have to show that for any $\vep>0$ there exists $\delta>0$ such that $m(A)<\delta$ implies $\int_{A}P_{\sigma^{-n}\omega}^{(n)}gdm<\vep$.

First we take $g'\in L^{\infty}_+(X,m)$ such that $\lV g-g'\rV_{L^1(X)}<\vep/3$.
Write for any $\alpha>0$
\begin{align*}
\int_{X\setminus F_{\alpha,n}}g'\left(P_{\sigma^{-n}\omega}^{(n)}\right)^*1_Adm
\le \lV g'\rV_{L^{\infty(X)}}m(X\setminus F_{\alpha,n})
\end{align*}
for any $A\in\mathcal{A}$.
Then by assumption we can find $\alpha>0$ so that for any $n\in\N$ we have $\lV g'\rV_{L^{\infty}(X)}m(X\setminus F_{\alpha,n})<\vep/3$.

On the other hand, we have
\begin{align*}
\int_{F_{\alpha,n}}g'\left( P_{\sigma^{-n}\omega}^{(n)}\right)^*1_Adm
&\le \lV g'\rV_{L^{\infty}(X)}\int_{F_{\alpha,n}}\left( P_{\sigma^{-n}\omega}^{(n)}\right)^*1_Adm\\
&< \lV g'\rV_{L^{\infty}(X)}\int_{F_{\alpha,n}}\frac{h_{\sigma^{-n}\omega}}{\alpha}\left( P_{\sigma^{-n}\omega}^{(n)}\right)^*1_Adm\\
&< \frac{\lV g'\rV_{L^{\infty}(X)}}{\alpha}\int_Ah_{\omega}dm.
\end{align*}
if we choose $\delta$ so that $m(A)<\delta$ implies $\int_Ah_{\omega}dm<\frac{\alpha\vep}{3\lV g'\rV_{L^{\infty}(X)}}$,
the last term in the above inequality is bounded above by $\vep/3$.

Therefore, we have
\begin{align*}
\int_A P_{\sigma^{-n}\omega}^{(n)}g'dm=\int_{F_{\alpha,n}\cup(X\setminus F_{\alpha,n})}g'\left( P_{\sigma^{-n}\omega}^{(n)}\right)^*1_Adm<\frac{2}{3}\vep
\end{align*}
and hence for any $\vep>0$ there exists $\delta>0$ such that $m(A)<\delta$ implies for any $n\in\N$,
\begin{align*}
\int_A P_{\sigma^{-n}\omega}^{(n)} gdm\le\int_A P_{\sigma^{-n}\omega}^{(n)}g'dm+\lV g-g'\rV_{L^1(X)}<\vep.
\end{align*}
This completes the proof.
\end{proof}

\section{Proof of Theorem \ref{MET}}

The key idea of the proof is based on the conventional mean ergodic theorem by Yosida and Kakutani \cite{YK}. We first prepare the following lemma.

\begin{lemma}\label{lemHB}
Assume that the sequence $\{\mathscr{A}^nf\}_n$ is fiberwise weakly precompact for $f\in L^1(\Omega,\mathfrak{X})$, that is, there exists  $h\in L^1(\Omega,\mathfrak{X})$ such that for $\mathbb{P}$-almost every $\omega\in\Omega$, there is a subsequence $\{n_k(\omega)\}_k$ for which $(\mathscr{A}^{n_k(\omega)}f)(\omega)$ weakly converges to $h(\omega)$. Then,
for any $\varepsilon>0$, there exist $g,r\in L^1(\Omega,\mathfrak{X})$ with
$\opnorm{r}_1<\varepsilon$ such that 
\begin{eqnarray}\label{HB}
f-h =\mathscr{P}g-g +r.
\end{eqnarray}
\end{lemma}

\begin{proof}
If this is not true, there exist $\varepsilon>0$ and $f\in L^1(\Omega,\mathfrak{X})$ such that for any $g,r\in L^1(\Omega,\mathfrak{X})$ with $\opnorm{r}_1<{\varepsilon}$, \eqref{HB} does not hold, in other words, there exist $\widetilde{\varepsilon}>0$ and a positive measure set $B$, $\mathbb{P}(B)>0$, such that for any $\omega\in B$,
\begin{align*}
\left\lVert(\mathscr{P}g-g)(\omega)-(f(\omega)-h(\omega))\right\rVert_{\mathfrak{X}}\geq\widetilde{\varepsilon},
\end{align*}
that is, $f(\omega)-h(\omega)\notin E(\omega)$  for any $\omega\in B$ where 
\begin{align*}
E(\omega)\coloneqq\text{cl}\left\{\left((\mathscr{P}-\mathscr{I})g\right)(\omega) : g\in L^1(\Omega,\mathfrak{X})\right\}
\end{align*}
where $\mathscr{I}$ is the identity operator on $L^1(\Omega,\mathfrak{X})$.
Indeed, if we assume $f(\omega)- h(\omega)\in E(\omega)$, there exists  a sequence $\{t_n\}_n\subset E(\omega)$ such that $\displaystyle\lim_{n\to\infty}t_n=f(\omega)- h(\omega)$. Namely, for any $\varepsilon>0$, there is $N$ such that if $n\geq N$ then 
\begin{align*}
\left\lVert (\mathscr{P}g-g)(\omega)-(f(\omega)- h(\omega))\right\rVert_{\mathfrak{X}}<\varepsilon,
\end{align*}
which is contradiction, and then $f(\omega)- h(\omega)\notin E(\omega)$.

Then, the Hahn--Banach theorem (Corollary 13, pp~63 in \cite{DuSch}) leads to that there exists $q\in \mathfrak{X}^*$ such that 
\begin{align}\label{HBeq1}
 \langle f(\omega)- h(\omega),q\rangle \neq 0
\end{align}
and $\langle k,q \rangle=0$ for any $k\in E(\omega)$.
Especially, for any $j=0,1,2,\cdots$, 
\begin{align*}
\left\langle \left((\mathscr{P}-\mathscr{I})\mathscr{P}^jf\right)(\omega),q\right\rangle=0,
\end{align*}
that is,
$\langle(\mathscr{P}^{j+1}f)(\omega),q\rangle=\langle(\mathscr{P}^{j}f)(\omega),q\rangle$. Thus, for any $N\in\mathbb{N}$,
\begin{align*}
\left\langle\left(\mathscr{A}^Nf\right)(\omega),q\right\rangle = \langle f(\omega),q\rangle
\end{align*}
Since $\{\mathscr{A}^nf\}_n$ is fiberwise weakly precompact,
\begin{eqnarray}
\lim_{k\to\infty} 
\left<\left(\mathscr{A}^{n_k(\omega)}f\right)(\omega),q\right>
=\left< h(\omega),q\right>\nonumber
\end{eqnarray}
On the other hand, 
\begin{align*}
\left<\left(\mathscr{A}^{n_k(\omega)}f\right)(\omega),q\right>
&=\frac{1}{n_k(\omega)}\sum_{t=0}^{n_k(\omega)-1}\left<(\mathscr{P}^t f)(\omega),q\right>\\
&=\frac{1}{n_k(\omega)}\sum_{t=0}^{n_k(\omega)-1}\left< f(\omega),q\right>
=\left< f(\omega),q\right>.\nonumber
\end{align*}
Therefore, we have $\langle f(\omega),q\rangle=\langle  h(\omega),q\rangle$  which contradicts to \eqref{HBeq1}. This completes the proof.
\end{proof}

\noindent
{\bf Proof of Theorem \ref{MET}}.
By \eqref{HB}, for $f\in L^1(\Omega,\mathfrak{X})$, there are $g,r\in L^1(\Omega,\mathfrak{X})$ with $\opnorm{r}_1<\varepsilon$ such that for every $k$,
\begin{align*}
\mathscr{P}^kf-\mathscr{P}^kh
=\mathscr{P}^{k+1}g-\mathscr{P}^kg+\mathscr{P}^kr.
\end{align*}
Taking the average of the above equation for $k$ from 0 to $n-1$,
\begin{align*}
\frac{1}{n}\sum_{k=0}^{n-1}\mathscr{P}^kf-\frac{1}{n}\sum_{k=0}^{n-1}\mathscr{P}^kh
=\frac{1}{n}\sum_{k=0}^{n-1}\left(\mathscr{P}^{k+1}g-\mathscr{P}^kg\right)+\frac{1}{n}\sum_{k=0}^{n-1}\mathscr{P}^kr,
\end{align*}
so that
\begin{align*}
\mathscr{A}^nf-\mathscr{A}^nh
=\frac{1}{n}\sum_{k=0}^{n-1}\left(\mathscr{P}^{k+1}g-\mathscr{P}^kg\right)+\mathscr{A}^nr.
\end{align*}
Hence, we have
\begin{align*}
& \left\lVert (\mathscr{A}^nf)(\omega)-(\mathscr{A}^nh)(\omega) \right\rVert_{\mathfrak{X}}\\
&\leq
\frac{1}{n}\left\lVert\sum_{k=0}^{n-1}\left((\mathscr{P}^{k+1}g)(\omega)-(\mathscr{P}^kg)(\omega)\right)\right\rVert_{\mathfrak{X}}
+
\left\lVert (\mathscr{A}^nr)(\omega)\right\rVert_{\mathfrak{X}}\nonumber\\
&\leq 
\frac{1}{n}\left\lVert
\sum_{k=0}^{n-1}P_{\sigma^{-k-1}\omega}^{k+1} g_{\sigma^{-k-1}\omega}
-
\sum_{k=0}^{n-1}P_{\sigma^{-k}\omega}^{k} g_{\sigma^{-k}\omega}
\right\rVert_{\mathfrak{X}}
+\frac{1}{n}\sum_{k=0}^{n-1} \left\lVert P^k_{\sigma^{-k}\omega}r_{\sigma^{-k}\omega}\right\rVert_{\mathfrak{X}}\nonumber\\
&\leq
\frac{1}{n}\left\lVert
P_{\sigma^{-n}\omega}^n g_{\sigma^{-n}\omega}-g_{\omega}
\right\rVert_{\mathfrak{X}}+\frac{1}{n}\sum_{k=0}^{n-1} \left\lVert r_{\sigma^{-k}\omega}\right\rVert_{\mathfrak{X}}.\nonumber
\end{align*}
The first term can be calculated as
\begin{align*}
\frac{1}{n}\left\lVert
P_{\sigma^{-n}\omega}^n g_{\sigma^{-n}\omega}-g_{\omega}
\right\rVert_{\mathfrak{X}}
&\leq
\frac{1}{n}\left(\left\lVert
P_{\sigma^{-n}\omega}^n g_{\sigma^{-n}\omega}\right\rVert_{\mathfrak{X}}
+\left\lVert g_{\omega}
\right\rVert_{\mathfrak{X}} \right)\nonumber\\
&\leq
\frac{1}{n}\left(\left\lVert
 g_{\sigma^{-n}\omega}\right\rVert_{\mathfrak{X}}
+\left\lVert g_{\omega}
\right\rVert_{\mathfrak{X}} \right)
\nonumber
\end{align*}
Since $\left\lVert g_{\omega}\right\rVert_{\mathfrak{X}}$ is bounded, $\left\lVert g_{\omega} \right\rVert_{\mathfrak{X}}/n\to0$ as $n\to\infty$. We must show that $\left\lVert g_{\sigma^{-n}\omega} \right\rVert_{\mathfrak{X}}/n\to0$ as $n\to\infty$. If it is not true, the sequence $\{\left\lVert g_{\sigma^{-n}\omega} \right\rVert_X\}_n$ must diverge as $n\to\infty$ so that the Ces\`{a}ro sum $\frac{1}{N}\sum_{n=0}^{N-1}\left\lVert g_{\sigma^{-n}\omega} \right\rVert_{\mathfrak{X}}$ also diverge as $N\to\infty$. However, by the Birkhoff individual ergodic theorem,
\begin{align*}
\lim_{N\to\infty}\frac{1}{N}\sum_{n=0}^{N-1}\left\lVert g_{\sigma^{-n}\omega} \right\rVert_{\mathfrak{X}}=\int_\Omega \left\lVert g_{\omega} \right\rVert_{\mathfrak{X}} d\mathbb{P}(\omega)=\opnorm{g}_1<\infty,
\end{align*}
which is contradiction. Hence, we have $\left\lVert g_{\sigma^{-n}\omega} \right\rVert_{\mathfrak{X}}/n\to0$ as $n\to\infty$. Moreover,  using again the Birkhoff individual ergodic theorem,
\begin{align*}
\lim_{n\to\infty}\frac{1}{n}\sum_{k=0}^{n-1} \left\lVert r_{\sigma^{-k}\omega}\right\rVert_{\mathfrak{X}}=\int_{\Omega}\left\lVert r_\omega\right\rVert_{\mathfrak{X}}d\mathbb{P}(\omega)=\opnorm{r}_1<\varepsilon.
\end{align*}
Therefore, we can derive
\begin{align}\label{eq4}
\lim_{n\to0}\left\lVert (\mathscr{A}^nf)(\omega)-(\mathscr{A}^nh)(\omega) \right\rVert_{\mathfrak{X}} =0.
\end{align}
Now, we use the similar argument with Lemma \ref{XX}. The boundedness and uniform integrability for $\{\mathscr{A}^nf\}_n$ are immediately shown. Moreover, since $\{\mathscr{A}^nf\}_n$ is fiberwise weakly precompact, that is, $\{(\mathscr{A}^nf)(\omega)\}_n$ is weakly precompact, then $\{\mathscr{A}^nf\}_n$ is weakly Cauchy in $L^1(\Omega,\mathfrak{X})$ from Corollary 2.4 (i) in \cite{G}. By the fact that $L^1(\Omega,\mathfrak{X})$ is weakly sequentially complete if $\mathfrak{X}$ is weakly sequentially complete \cite{Ta}, we have $\{\mathscr{A}^nf\}_n$ is weakly precompact in $L^1(\Omega,\mathfrak{X})$.  Then there exists subsequence $\{n_k\}_k$ such that $\mathscr{A}^{n_k}f$ converges weakly to $h$, and $\mathscr{A}^{n_k}\mathscr{P}f$ converges to $\mathscr{P}h$ since $\mathscr{A}^{n_k}\mathscr{P}f=\mathscr{P}\mathscr{A}^{n_k}f$. Furthermore, 
the calculation
\begin{eqnarray}
\opnorm{\mathscr{A}^nf-\mathscr{A}^n\mathscr{P}f}_1
&=&\frac{1}{n}\opnorm{f-(\mathscr{P}^{n}f)}_1\nonumber\\
&\leq& \frac{1}{n}(\opnorm{f}_1+\opnorm{\mathscr{P}^nf}_1)\nonumber\\
&\leq& \frac{2}{n}\opnorm{f}_1\to0\quad(n\to\infty)\nonumber
\end{eqnarray}
implies $\mathscr{P}h=h$ so that $\mathscr{A}^nh=h$. Then \eqref{eq4} becomes
\begin{align*}
\lim_{n\to0}\left\lVert (\mathscr{A}^nf)(\omega)-h(\omega) \right\rVert_{\mathfrak{X}} =0.
\end{align*}
This completes the proof.

\section*{Acknowledgments}
The authors would like to express our gratitude to Professor Yushi Nakano (Tokai university) for giving us constructive suggestions and warm encouragement.
The authors also thank the anonymous referee for valuable comments which improve the work.
Fumihiko Nakamura was  supported by JSPS KAKENHI Grant Number 	  19K21834.  
Hisayoshi Toyokawa was  supported by JSPS KAKENHI Grant Numbers 19K21834 and 21K20330.

  %%%%% References

\end{document}